\documentclass[11pt]{article}
\usepackage[margin=1in]{geometry}
\usepackage{amsmath,amssymb,amsthm,mathtools}
\usepackage{enumitem}
\usepackage{hyperref}
\hypersetup{
  hidelinks,
  pdftitle={Layerwise Terminal Discrepancy for Reverse Heat on the Boolean Cube},
  pdfauthor={Yanjin Xiang and Zhihua Zhang},
  pdfsubject={A layerwise refinement for Talagrand's Boolean convolution problem},
  pdfkeywords={Talagrand convolution conjecture, Boolean heat semigroup, reverse heat, terminal testing discrepancy, anti-concentration}
}
\usepackage[nameinlink,noabbrev,capitalise]{cleveref}

\newtheorem{theorem}{Theorem}[section]
\newtheorem{proposition}[theorem]{Proposition}
\newtheorem{lemma}[theorem]{Lemma}
\newtheorem{corollary}[theorem]{Corollary}

\crefname{theorem}{Theorem}{Theorems}
\crefname{proposition}{Proposition}{Propositions}
\crefname{lemma}{Lemma}{Lemmas}
\crefname{corollary}{Corollary}{Corollaries}
\crefname{definition}{Definition}{Definitions}
\Crefname{theorem}{Theorem}{Theorems}
\Crefname{proposition}{Proposition}{Propositions}
\Crefname{lemma}{Lemma}{Lemmas}
\Crefname{corollary}{Corollary}{Corollaries}
\Crefname{definition}{Definition}{Definitions}

\AddToHook{env/theorem/begin}{\crefalias{section}{theorem}\crefalias{theorem}{theorem}}
\AddToHook{env/proposition/begin}{\crefalias{section}{proposition}\crefalias{theorem}{proposition}}
\AddToHook{env/lemma/begin}{\crefalias{section}{lemma}\crefalias{theorem}{lemma}}
\AddToHook{env/corollary/begin}{\crefalias{section}{corollary}\crefalias{theorem}{corollary}}
\AddToHook{env/assumption/begin}{\crefalias{section}{assumption}\crefalias{theorem}{assumption}}
\AddToHook{env/remark/begin}{\crefalias{section}{remark}\crefalias{theorem}{remark}}
\AddToHook{env/definition/begin}{\crefalias{section}{definition}\crefalias{theorem}{definition}}

\newcommand{\cube}{\{-1,1\}^n}

\newcommand{\E}{\mathbb E}
\newcommand{\Pp}{\mathbb P}
\newcommand{\R}{\mathbb R}

\newcommand{\1}{\mathbf 1}

\newcommand{\TV}{\operatorname{TV}}
\newcommand{\Law}{\operatorname{Law}}

\newcommand{\cA}{\mathcal A}

\newcommand{\cL}{\mathcal L}
\newcommand{\cG}{\mathcal G}
\newcommand{\cB}{\mathcal B}

\newcommand{\cS}{\mathcal S}
\newcommand{\eps}{\varepsilon}
\newcommand{\To}{{T_{\circ}}}
\newcommand{\StopT}{{\sigma_\theta}}

\newcommand{\bdel}{\bar\delta}
\newcommand{\flip}{\sigma}
\newcommand{\abs}[1]{\left|#1\right|}

\newcommand{\set}[1]{\left\{#1\right\}}

\renewcommand*{\d}{\mathop{}\!\mathrm{d}}	% differential

\title{Layerwise Terminal Discrepancy in Chen's Reverse-Heat Coupling on the Boolean Cube}
\author{
Yanjin Xiang\\
Peking University\\
\texttt{2401110086@stu.pku.edu.cn}
\and
Zhihua Zhang\\
Peking University\\
\texttt{zhzhang@math.pku.edu.cn}
}
\date{\today}

\begin{document}
\maketitle

\begin{abstract}
Recently, Chen \cite{Chen2026} proved that Talagrand's Boolean convolution conjecture
holds up to the dimension-free factor \((\log\log\eta)^{3/2}\), namely for every fixed \(\tau>0\),
\[
  \mu\{P_\tau f>\eta\|f\|_1\}
  \le C_\tau
  \frac{(\log\log\eta)^{3/2}}{\eta\sqrt{\log\eta}},
  \qquad \eta>e^3.
\]
We revisit the terminal testing-discrepancy step in Chen's perturbed
reverse-heat coupling.  Chen estimates this discrepancy globally in terms
of the remaining gap to the terminal level.  We keep the same coupling and
the same reverse-heat formulations, but localize the terminal discrepancy on each
remaining-gap layer before summing the layers.  This changes the fixed-time
anti-concentration cost from order
\((\log L)^{3/2}/\sqrt L\) to order \((\log L)/\sqrt L\), where
\(L=\log\eta\).  Consequently, we obtain a \((\log\log\eta)^{1/2}\) improvement as
\[
  \mu\{P_\tau f>\eta\|f\|_1\}
  \le C_\tau
  \frac{\log\log\eta}{\eta\sqrt{\log\eta}},
  \qquad \eta>e^3.
\]
\end{abstract}

\tableofcontents

\section{Introduction}

Talagrand's convolution conjecture
\cite{Talagrand1989,Talagrand2016} asks for the sharp dimension-free
regularization of nonnegative \(L^1\) functions $f\colon \{-1,1\}^n \to [0, \infty)$ under convolution by a
biased coin on the Boolean cube. In particular, given $t\ge 0$, let \(\xi_t=(\xi_t^{(1)}, \ldots,  \xi_t^{(n)})^T \) be a biased random vector on $\{-1, 1\}^n$ whose entries $\xi_t^{(i)}$ are independent and identically distributed (iid) and satisfy 
\[
\Pp(\xi_t^{(i)}=1)=\frac{1+e^{-t}}2, \qquad \Pp(\xi_t^{(i)}=-1)=\frac{1-e^{-t}}2.
\]
Let \((P_t)_{t\ge0}\) be the Boolean heat
semigroup
\[
  P_t f(x)=\E f(x\odot \xi_t),
\]
where $\odot$ denotes element-wise products.  
Let \(\mu\) be the uniform probability
measure on \(\cube\). For $f\colon \{-1,1\}^n \to \mathbb{R}$, define its $L_1(\mu)$-norm as $\|f\|_1:= \int{|f| \d \mu}$, and denote $\mu\{P_t f>\eta\|f\|_1\} := \Pp_{X \sim \mu} \{P_t f(X)>\eta\|f\|_1\}$.   

Talagrand conjectured that, for every fixed \(\tau>0\),
\[
  \mu\{P_\tau f>\eta\|f\|_1\} 
  \le \frac{C_\tau}{\eta\sqrt{\log\eta}},
  \qquad \eta>1,
\]
uniformly over the dimension and over \(f\ge0\). The conjectured order is
known to be optimal up to  the constant \(C_\tau\).

The Gaussian analogue, with the Ornstein--Uhlenbeck semigroup replacing
the Boolean heat semigroup, was initiated by Ball, Barthe, Bednorz,
Oleszkiewicz and Wolff \cite{BallBartheBednorzOleszkiewiczWolff2013} and
was resolved through the stochastic F\"ollmer-process approach of
Eldan--Lee \cite{EldanLee2018} and the refinement of Lehec
\cite{Lehec2016}.  In the Boolean setting in question,  Chen \cite{Chen2026}
recently proved the first dimension-free logarithmic improvement over
Markov's inequality:
\[
  \mu\{P_\tau f>\eta\|f\|_1\}
  \le C_\tau
  \frac{(\log\log\eta)^{3/2}}{\eta\sqrt{\log\eta}}, \quad \eta> e^3.
\]

Chen's seminal work \cite{Chen2026} has three main ingredients.  First, the weak-type estimate is reduced
to a fixed-time anti-concentration estimate for
\[
  \cA_t(I)=\nu_t\{\log P_t f\in I\},\qquad d\nu_t=P_t f\,d\mu .
\]
Second,  a reverse-heat process \(V_t\) is constructed, while from a start
time \(\theta\) a coupled process \(W_t\) whose jump rates are perturbed
until a terminal stopping time.  This coupling gives a global terminal
testing estimate recalled in \cref{app:chen-lemma2}, and an approximate
monotonicity statement at the tail level recalled in
\cref{lem:monotone}.  Third, the time-smoothed profile bound is proven, which is
recalled in \cref{app:chen-lemma4}.  Averaging over the start time
\(\theta\in[\To-1,\To]\) then controls the remaining-gap quantity appearing
in the global terminal estimate, which yields the
\((\log\log\eta)^{3/2}\) factor.

Our present paper refines only the terminal testing-discrepancy part of the above
scheme.  Chen \cite{Chen2026} used the global estimate, in which the main term has the form
\[
  \alpha
  \sqrt{\E\frac{\1_{\{R_\theta\ge \alpha\}}}{R_\theta+1}},
  \qquad
  L=\log\eta,\quad \alpha=\frac12\log L+1.
\]
After time averaging, this term is of size
\((\log L)^{3/2}/\sqrt L\). Alternatively, we keep an
\(\mathcal F_\theta\)-measurable localization throughout the Duhamel
identity, the terminal bridge expansion, and the Doob-transform energy
estimate.  Applying the resulting estimate to
\[
  G_r(\theta) :=\{r\le R_\theta<r+1\}
\]
gives the layerwise contribution
\[
  \lesssim_\tau
  \left(\frac{\alpha}{\sqrt r}+\frac{\alpha^2}{r}\right)
  \Pp(G_r(\theta)).
\]
The same time-smoothed profile estimate as in \cite{Chen2026}  now gives
\[
  \frac{\alpha}{L}\sum_{\alpha\le r\le L/2}r^{-1/2}
  +\frac{\alpha^2}{L}\sum_{\alpha\le r\le L/2}r^{-1}
  \lesssim \frac{\alpha}{\sqrt L}.
\]
This proves the weak-type bound stated in \cref{thm:main}, which will be presented in Section~\ref{sec:main-results}. 

% In Section~\ref{sec:preliminary} we give some preliminaries, including elementary heat estimates and time-%smoothed profile estimate.  In Section~\ref{sec:main-results} we give the improved results via  the %layerwise localization.

\subsection{Preliminaries}
\label{sec:preliminary}

%We first give some preliminaries, including elementary heat estimates and time-smoothed profile estimate. 
Before presenting our main result, we first give some  preliminaries, including multilinear extension of a function on $\{-1, 1\}^n$,
{elementary heat estimates} and {time-smoothed profile estimate}.

\subsubsection*{Multilinear extension of functions on $\{-1, 1\}^n$}

For \(z\in[-1,1]^n\), let \(\mu_z\) be the product probability measure on
\(\cube\) whose \(i\)-th coordinate has mean \(z_i\), that is,
\[
  \mu_z\{X_i=1\}=\frac{1+z_i}{2},
  \qquad
  \mu_z\{X_i=-1\}=\frac{1-z_i}{2}.
\]
If \(h:\cube\to\R\), its multilinear extension is the polynomial
\[
  H(z)
  =
  \E_{\mu_z}h(X)
  =
  \sum_{x\in\cube}
  h(x)\prod_{i=1}^n\frac{1+x_i z_i}{2},
  \qquad z\in[-1,1]^n.
\]
This is the unique polynomial that is affine in each coordinate separately
and agrees with \(h\) on the vertices of the cube.  Throughout the paper,
when a function originally defined on \(\cube\) is evaluated at a vector in
\([-1,1]^n\), we regard this as its multilinear extension.  We also write
\(\flip_i z\) for the vector obtained from \(z\) by changing the sign of its
\(i\)-th coordinate.

The Boolean heat semigroup is compatible with this convention.  If \(H\)
denotes the multilinear extension of \(h\), then for \(x\in\cube\),
\[
  P_t h(x)=H(e^{-t}x).
\]
Indeed, under the heat semigroup generated by
\(\frac12\sum_i(h(\flip_i x)-h(x))\), each coordinate keeps the mean
\(e^{-t}x_i\).  In particular, if \(h\ge0\) and \(h\not\equiv0\), then
\(H(z)>0\) for every \(z\in(-1,1)^n\); if \(0\le h\le1\) on the cube, then
\(0\le H(z)\le1\) on \([-1,1]^n\).

We shall also use the following elementary consequence frequently.  If
\(\phi\) is multilinear and \(Y_1,\ldots,Y_n\) are independent
\(\{-1,1\}\)-valued random variables with
\(\E Y_i=m_i\), then
\[
  \E\,\phi(Y)=\phi(m_1,\ldots,m_n).
\]
This is the form used later when the terminal Boolean bridge is replaced by
the product mean \(m_t(x,y,\zeta)\).

\subsubsection*{Elementary heat estimates}

\begin{lemma}[Edge-ratio bound]\label{lem:edge-ratio}
Let \(f\ge0\) and \(f\not\equiv0\), and define $g_t=\log (P_t f)$. Then for \(t>0\), \(x\in\cube\), and
\(i\in[n]\), we have
\[
  \frac{1-e^{-t}}{1+e^{-t}}
  \le
  \frac{P_t f(\flip_i x)}{P_t f(x)}
  \le
  \frac{1+e^{-t}}{1-e^{-t}}.
\]
Subsequently, we have 
\[
  |g_t(\flip_i x)-g_t(x)|
  \le
  \log\frac{1+e^{-t}}{1-e^{-t}}.
\]
\end{lemma}

\begin{proof}
This is the edge-ratio part of Chen \cite[Lemma~5]{Chen2026}, recalled in
\cref{app:chen-lemma5}, applied to Chen's multilinear extension at the point
\(z=e^{-t}x\).  The logarithmic bound follows immediately after taking
logarithms.
\end{proof}

\begin{lemma}[Level-one inequality]\label{lem:level-one}
Let \(h \colon \cube\to\{0,1\}\), and let \(H\) be its multilinear extension.
For every \(z\in(-1,1)^n\), then we have
\[
  \sum_{i=1}^n(1-z_i^2)(\partial_i H(z))^2
  \le H(z)-H(z)^2.
\]
\end{lemma}

\begin{proof}
This is the standard level-one inequality for biased Fourier analysis
\cite[Chapter~8]{ODonnell2014}.
Let \(\mu_z\) be the product measure with coordinate means \(z_i\).  The
biased Fourier expansion gives
\[
  (1-z_i^2)^{1/2}\partial_i H(z)
  =
  \widehat h_{\mu_z}(\{i\}).
\]
Summing over \(i\) and using Parseval gives the variance
\(\operatorname{Var}_{\mu_z}(h)=H(z)-H(z)^2\).
\end{proof}

\subsubsection*{Time-smoothed profile estimate}

\begin{lemma}[Time-smoothed anti-concentration]\label{lem:time-profile}
Let \(f:\{-1,1\}^n\to(0,\infty)\) satisfy \(\|f\|_1=1\). Put
\[
  u_s=P_s f,\qquad g_s=\log u_s,\qquad d\nu_s=u_s\,d\mu,
  \qquad \cA_s(I)=\nu_s\{g_s\in I\}.
\]
Then for every \(\ell>2\),
\[
  \int_0^\infty \cA_s((\ell,\ell+1])\,ds
  \le \frac{C}{\ell},
\]
where \(C\) is universal.
\end{lemma}

\begin{proof}
This is Chen's time-smoothed anti-concentration profile estimate
\cite[Lemma~4]{Chen2026}, recalled in \cref{app:chen-lemma4}, after the
notational identification above.
\end{proof}

\subsection{The improved result and reduction to anti-concentration}
\label{sec:main-results}

Here and later, constants denoted by \(C\) are universal, while constants denoted by
\(C_\tau\) may depend on \(\tau\), but never on
\(n,f,L,\eta,T\), or \(\theta\); both may change from line to line.

Let \(P_t=e^{t\cL}\), where
\[
  \cL h(x)=\frac12\sum_{i=1}^n\bigl(h(\flip_i x)-h(x)\bigr).
\]
We write \(u_t=P_t f\), \(g_t=\log u_t\), and
\[
  d\nu_t=u_t\,d\mu,
  \qquad
  \cA_t(I)=\nu_t\{g_t\in I\}.
\]
By homogeneity we assume \(\|f\|_1=1\).  The stochastic notation below is
written for strictly positive \(f\), so that \(g_t\) and the reverse score
are everywhere defined.  If \(f\equiv0\), the argument  is trivial.  For a
general nonnegative \(f\not\equiv0\), apply the argument to
\[
  F_\eps=\frac{f+\eps}{\|f+\eps\|_1}.
\]
Then \(P_\tau F_\eps(x)\to P_\tau f(x)/\|f\|_1\) for every \(x\).  Because
\(\{P_\tau f/\|f\|_1>\eta\}\subseteq\liminf_{\eps\downarrow0}
\{P_\tau F_\eps>\eta\}\), Fatou's lemma transfers the weak-type estimate
from \(F_\eps\) to \(f\). We now state our improved result in the following theorem.

\begin{theorem}[Layerwise reverse-heat bound]\label{thm:main}
For every \(\tau>0\) there exists
\(C_\tau<\infty\) such that, for every \(n\ge1\), every
\(f:\cube\to\R_+\), and every \(\eta>e^3\),
\[
  \mu\{x:P_\tau f(x)>\eta\|f\|_1\}
  \le C_\tau
  \frac{\log\log\eta}{\eta\sqrt{\log\eta}}.
\]
\end{theorem}

Theorem~\ref{thm:main} can be reduced to the following fixed-time anti-concentration proposition, whose proof is built on our layerwise localization approach to the terminal testing estimate. 
We will give the layerwise localization approach in
\cref{sec:localized-discrepancy}, and then prove Proposition~\ref{prop:AC}  in Section~\ref{sec:proof-fixed-time}.

\begin{proposition}[Fixed-time anti-concentration]\label{prop:AC}
For every \(\tau>0\) there exists \(C_\tau<\infty\) such that, for every \(L\ge8\),
\[
  \cA_\tau((L,L+1])
  \le C_\tau \frac{\log L}{\sqrt L}.
\]
\end{proposition}

\begin{proof}[To prove \cref{thm:main} from \cref{prop:AC}]
Put \(L=\log\eta\).  For \(L\ge8\), then
\begin{align*}
  \mu\{g_\tau>L\}
  &\le \sum_{k=0}^{\infty}
  \mu\{L+k<g_\tau\le L+k+1\}        \\
  &\le \sum_{k=0}^{\infty}
  e^{-(L+k)}\nu_\tau\{L+k<g_\tau\le L+k+1\} \\
  &\le C_\tau e^{-L}
  \sum_{k=0}^{\infty}e^{-k}\frac{\log(L+k)}{\sqrt{L+k}}
  \le C_\tau e^{-L}\frac{\log L}{\sqrt L}.
\end{align*}
Since \(e^L=\eta\), this proves the desired estimate for
\(\eta>e^8\).  The range \(e^3<\eta\le e^8\) is absorbed by Markov's
inequality after changing \(C_\tau\).
\end{proof}

\paragraph{The Paper Organization} \;
We use the same
reverse-heat formulations as in \cite{Chen2026}, and introduce them in Section~\ref{sec:reverse-heat}. 
We present our layerwise localization approach in
\cref{sec:localized-discrepancy}, and then complete the proof of Proposition~\ref{prop:AC}  in Section~\ref{sec:proof-fixed-time}.
In Section~\ref{sec:discussion} we explain where the remaining logarithm enters and why
removing it would require input beyond the present layerwise
Cauchy--Schwarz/profile argument.
For self-contained purpose, we also present some auxiliary lemmas given in \cite{Chen2026} in the appendix. 
%The layerwise localization and the summation argument are presented in
%\cref{sec:localized-discrepancy,sec:proof-fixed-time}.

\section{Reverse-heat coupling}\label{sec:reverse-heat}

%Throughout \cref{sec:reverse-heat,sec:localized-discrepancy} and the proof of \cref{prop:AC}, 
Here and later, we work with \(f>0\) and \(\|f\|_1=1\).  The
general nonnegative case is obtained by the approximation explained in
\cref{thm:main}.

Fix \(\tau>0\), choose \(T>\tau+1\), and set \(\To=T-\tau\).  The reverse
heat process \((V_t)_{0\le t\le T}\) is the time reversal of the forward
heat chain started from \(f\,d\mu\).  Thus \(V_t\sim\nu_{T-t}\), and its
generator is
\[
  \widetilde{\cL}_t h(x)
  =
  \frac12\sum_{i=1}^n
  \frac{P_{T-t}f(\flip_i x)}{P_{T-t}f(x)}
  \bigl(h(\flip_i x)-h(x)\bigr).
\]
It is convenient to write
\[
  S_i(t,x)
  =
  e^{-(T-t)}
  \frac{x_i\,\partial_i f(e^{-(T-t)}x)}
       {f(e^{-(T-t)}x)}.
\]
Then
\[
  \frac{P_{T-t}f(\flip_i x)}{P_{T-t}f(x)}
  =
  1-2S_i(t,x),
\]
and the \(i\)-th reverse jump rate is \(1/2-S_i(t,x)\).  The factor
\(e^{-(T-t)}\) is part of the definition of the score.  With this
normalization, the edge-ratio bound gives
\(1-2S_i(t,x)\in[C_\tau^{-1},C_\tau]\) for \(t\le\To\).

Let \(L\ge8\), and set
\[
  \alpha=\frac12\log L+1.
\]
For a perturbation start time \(\theta\in[\To-1,\To)\), define the
remaining gap
\[
  R_\theta=[L-g_{T-\theta}(V_\theta)]_+.
\]
The perturbation amplitude is the \(\mathcal F_\theta\)-measurable
random variable
\[
  \bdel=\frac{\alpha\,\1_{\{R_\theta\ge \alpha\}}}{R_\theta+1}.
\]
The perturbed process \(W\) starts from \(W_\theta=V_\theta\), is driven by
the same Poisson clocks as \(V\), and has coordinate perturbation
\(\delta_i(t)=\delta_i(t,V_{t-})\), where
\[
  \delta_i(t,x)
  =
  \bdel\left[
  \1_{\{S_i(t,x)>0\}}
  +
  \frac{1-2S_i(t,x)}{1-2\bdel S_i(t,x)}
  \1_{\{S_i(t,x)\le0\}}
  \right].
\]
The perturbation is stopped at
\[
  \StopT
  =
  \inf\set{t\in[\theta,\To] \colon
  \max\{g_{T-t}(V_t)-\alpha,\ g_{T-t}(W_t)\}\ge L}
  \wedge \To.
\]
The process \(t\mapsto\1_{\{t\le\StopT\}}\) is left-continuous and
adapted, hence predictable.
All predictable statements are with respect to the natural filtration
\[
  \mathcal F_t=\sigma(V_s \colon s\le t)\vee\sigma(W_s\colon \theta\le s\le t),
  \qquad t\in[\theta,\To],
\]
completed in the usual way. Under conditioning on \(V_T=\zeta\), we use
the same raw filtration on \([\theta,\To]\), completed under the
conditioned law.
Conditioned on \(\mathcal F_\theta\), the joint process \((V_t,W_t)\) is a
finite-state pure-jump process with predictable generator
\[
  \bar{\cL}_t^\delta h(x,y)
  =
  \bar{\cL}_t^0h(x,y)+\1_{\{t\le\StopT\}}\cB_t h(x,y),
\]
where
\[
  \Delta_i^yh(x,y)=h(x,\flip_i y)-h(x,y),\qquad
  \Delta_i^{xy}h(x,y)=h(\flip_i x,\flip_i y)-h(x,y),
\]
and, writing \(S_i=S_i(t,x)\), \(\delta_i=\delta_i(t,x)\),
\begin{align}
  \label{eq:joint-generator}
  \bar{\cL}_t^0h(x,y)
  &=
  \frac12\sum_{i=1}^n(1-2S_i)\Delta_i^{xy}h(x,y), \notag\\
  \cB_t h(x,y)
  &=
  \sum_{i=1}^n \1_{\{S_i>0\}}\delta_iS_i\Delta_i^yh(x,y)
  +
  \sum_{i=1}^n \1_{\{S_i\le0\}}\delta_iS_i
  \Delta_i^yh(\flip_i x,y).
\end{align}
The second line is a signed perturbation of the synchronized generator,
not a generator by itself; the full operator \(\bar{\cL}_t^\delta\) is the
predictable generator of the coupled process.
Throughout the sequel, \(T\) denotes the fixed terminal horizon in the
reverse construction, while \(\StopT\) denotes this stopping time.  The
perturbation part of the generator is always multiplied by
\(\1_{\{t\le\StopT\}}\), so integrals of perturbative terms may be written
over \([\theta,\StopT]\) or over \([\theta,\To]\) with this indicator.

\begin{lemma}[Joint-filtration martingale problem for the perturbed coupling]\label{lem:joint-filtration}
Consider the perturbed reverse-heat coupling constructed in
\cite[Section~3.1]{Chen2026}, written in the notation above; the
well-posedness of the coupled process is Chen's \cite[Lemma~6]{Chen2026},
recalled in \cref{app:chen-lemma6}.
For every strictly positive normalized \(f\), every \(L\ge8\), every
\(T>\tau+1\), and every \(\theta\in[\To-1,\To)\), the coupled process is
well-defined on
\([\theta,\To]\), has predictable generator \eqref{eq:joint-generator},
and makes the \(V\)-coordinate a time-inhomogeneous Markov chain with
generator \(\widetilde{\cL}_t\) with respect to the joint filtration
\((\mathcal F_t)\).  Equivalently, for every bounded \(h=h(x)\),
\[
  h(V_t)-h(V_\theta)
  -\int_\theta^t\widetilde{\cL}_s h(V_s)\,ds
\]
is an \((\mathcal F_t)\)-martingale.  In particular, for every bounded
test function depending only on \(x\), the perturbation operator
\(\cB_t\) vanishes.  Moreover, for every \(t\in[\theta,\To]\), if
\(P^V_{t,T}\) denotes the transition operator of the original reverse-heat
\(V\)-coordinate from \(t\) to the fixed terminal horizon \(T\), then for
every bounded terminal test \(\Phi\),
\[
  \E[\Phi(V_T)\mid\mathcal F_t]=P^V_{t,T}\Phi(V_t).
\]
\end{lemma}

\begin{proof}
This is the predictable-generator formulation of Chen's perturbed
reverse-heat coupling in \cite[Section~3.1]{Chen2026}, translated to the
present notation.  Chen's \cite[Lemma~5]{Chen2026}, recalled in
\cref{app:chen-lemma5}, gives the edge-ratio bounds used to keep the reverse
jump rates controlled, and Chen's \cite[Lemma~6]{Chen2026}, recalled in
\cref{app:chen-lemma6}, gives existence and uniqueness of the coupled SDE and
the bound on the perturbation size.  The stopped predictable joint generator
is Chen's generator representation \cite[Eq.~(19)--(20)]{Chen2026}, recalled
in \cref{app:chen-generator}:
under the notational translation displayed below, it is exactly
\eqref{eq:joint-generator}.

For tests \(h=h(x)\) depending only on the \(V\)-coordinate,
\(\Delta_i^yh(x,y)=0\) and
\(\Delta_i^yh(\flip_i x,y)=0\), so \(\cB_t h=0\).  Moreover
\(\Delta_i^{xy}h(x,y)=h(\flip_i x)-h(x)\), and therefore the
\(V\)-coordinate keeps the reverse-heat martingale problem with generator
\(\widetilde{\cL}_t\) with respect to the joint filtration.

It remains to identify the terminal transition.  Let \(\Phi\) be a
bounded terminal test and put
\[
  u(t,x)=P^V_{t,T}\Phi(x).
\]
On the finite state space, \(u\) is the unique bounded solution of
\[
  (\partial_t+\widetilde{\cL}_t)u=0,\qquad u(T,x)=\Phi(x).
\]
Applying the joint-filtration martingale problem to the time-dependent
test \(u(t,\cdot)\) on the interval \([\theta,\To]\) gives, for
\(t\le\To\),
\[
  \E\!\left[P^V_{\To,T}\Phi(V_\To)\mid\mathcal F_t\right]
  =
  P^V_{t,T}\Phi(V_t).
\]
After \(\To\), the \(V\)-coordinate is continued as the original
reverse-heat chain.  Conditional on \(V_\To\), its future driving clocks are
independent of \(\mathcal F_\To\), hence
\[
  \E[\Phi(V_T)\mid\mathcal F_\To]
  =
  P^V_{\To,T}\Phi(V_\To).
\]
The tower property therefore gives
\[
  \E[\Phi(V_T)\mid\mathcal F_t]
  =
  P^V_{t,T}\Phi(V_t),\qquad t\in[\theta,\To].
\]
Taking \(\Phi=\1_{\{\zeta\}}\) gives
\[
  H_t^\zeta(V_t)=\Pp(V_T=\zeta\mid\mathcal F_t),
\]
the form used in \cref{lem:terminal-conditioning}.
\end{proof}

We quote the needed reverse-heat estimates from \cite{Chen2026} in the
notation used below.  The translation is:
\[
\begin{array}{c|c}
\text{Chen notation} & \text{notation in this note}\\
\hline
\log \eta & L\\
\alpha=\frac12\log\log\eta+1 & \alpha=\frac12\log L+1\\
R_\theta=[\log\eta-\log P_{T-\theta}f(V_\theta)]_+
  & R_\theta=[L-g_{T-\theta}(V_\theta)]_+\\
S_i(\rho_t\mathbf V_t) & S_i(t,V_t)\\
\text{Chen's stopping time } \boldsymbol{\mathfrak T} & \sigma_\theta\\
\kappa=(1+e^{-\tau})/(1-e^{-\tau}) & \text{absorbed into }C_\tau
\end{array}
\]
The symbol \(\boldsymbol{\mathfrak T}\) in the table denotes Chen's
stopping time, not the terminal horizon \(T\) used here.  The score row
means explicitly
\[
  S_i(t,V_t)=
  e^{-(T-t)}
  \frac{V_t^{(i)}\,\partial_i f(e^{-(T-t)}V_t)}
      {f(e^{-(T-t)}V_t)}.
\]

The next two estimates are quoted directly from Chen
\cite[Lemmas~3 and~8]{Chen2026} after the notational translation displayed
above.

\begin{lemma}[Approximate monotonicity, Chen {\cite[Lemma~3]{Chen2026}}]\label{lem:monotone}
For the coupling above,
\[
  \Pp\{g_\tau(V_\To)>L+1\}
  \ge
  \Pp\{g_\tau(W_\To)>L\}
  -\cA_{T-\theta}((L-\alpha,L+\alpha])
  -\frac{3}{\sqrt L}.
\]
\end{lemma}

\begin{lemma}[Conditional score energy, Chen {\cite[Lemma~8]{Chen2026}}]\label{lem:score-energy}
Conditionally on \(\mathcal F_\theta\),
\[
  \E\left[
  \int_\theta^{\StopT}
  \sum_{i=1}^n S_i(t,V_{t-})^2\,dt
  \,\middle|\,\mathcal F_\theta
  \right]
  \le C_\tau(R_\theta+\alpha+1).
\]
In particular, for every \(E\in\mathcal F_\theta\),
\[
  \E\left[\1_E\bdel^2
  \int_\theta^{\StopT}\sum_{i=1}^n S_i(t,V_{t-})^2\,dt\right]
  \le C_\tau\,
  \E\left[\1_E\bdel^2(R_\theta+\alpha+1)\right].
\]
\end{lemma}

\begin{proof}[Comment on the conditional form]
We use Chen's estimate in the \(\mathcal F_\theta\)-conditional form
displayed above.  At the starting time \(\theta\), the coupling is
initialized by \(W_\theta=V_\theta\), and the perturbation is switched on
only after \(\theta\); hence \(\mathcal F_\theta\) is precisely the initial
information for Chen's stopped coupling.  Since both \(E\) and \(\bdel\)
are \(\mathcal F_\theta\)-measurable, multiplying the conditional estimate
by \(\1_E\bdel^2\) and taking expectations gives the displayed
consequence.
\end{proof}

\section{Localized terminal testing discrepancy}\label{sec:localized-discrepancy}

The main estimate in this section localizes the terminal
testing-discrepancy calculation.  The argument keeps an
\(\mathcal F_\theta\)-measurable factor throughout the Duhamel identity,
the bridge expansion, and the Doob-transform energy estimate.

\begin{lemma}[Boolean heat bridge algebra]\label{lem:bridge-algebra}
Fix \(t\in[\theta,\To]\) and condition on \(V_T=\zeta\).  Put
\[
  \gamma_t=e^{-(\To-t)},\qquad
  \beta=e^{-\tau},\qquad
  \rho_t=\gamma_t\beta=e^{-(T-t)}.
\]
Define
\[
  \lambda_{t,i}^{\zeta}(x)
  =
  \frac{1-\rho_t x_i\zeta_i}{1+\rho_t x_i\zeta_i}.
\]
Let
\[
  a_t=\frac{\gamma_t(1-\beta^2)}{1-\rho_t^2}
  =\frac{\sinh(T-\To)}{\sinh(T-t)},\qquad
  b_t=\frac{\beta(1-\gamma_t^2)}{1-\rho_t^2}
  =\frac{\sinh(\To-t)}{\sinh(T-t)},\qquad
  \omega_i=x_iy_i\zeta_i,
\]
and
\[
  m_t^{[i]}(x,y,\zeta)=a_ty_i+b_t\omega_i.
\]
Equivalently,
\[
  m_t^{[i]}(x,y,\zeta)
  =
  y_i\frac{\gamma_t+\beta x_i\zeta_i}{1+\rho_t x_i\zeta_i}.
\]
For a multilinear extension \(\phi\) of a \(\{0,1\}\)-valued function, set
\[
  q_t^\zeta(x,y)=\phi(m_t(x,y,\zeta)).
\]
Here \(\Delta_i^y h(x,y)=h(x,\flip_i y)-h(x,y)\) and
\(\Delta_i^{xy}h(x,y)=h(\flip_i x,\flip_i y)-h(x,y)\).
Then
\begin{align}
  \label{eq:bridge-delta-y}
  \Delta_i^yq_t^\zeta(x,y)
  &=-2(a_ty_i+b_t\omega_i)\partial_i\phi(m_t(x,y,\zeta)),\\
  \label{eq:bridge-delta-y-flipx}
  \Delta_i^yq_t^\zeta(\flip_i x,y)
  &=-2(a_ty_i-b_t\omega_i)\partial_i\phi(m_t(x,y,\zeta)),\\
  \label{eq:bridge-delta-xy}
  \Delta_i^{xy}q_t^\zeta(x,y)
  &=-2a_ty_i\partial_i\phi(m_t(x,y,\zeta)).
\end{align}
Moreover,
\begin{align}
  \label{eq:bridge-b-control}
  \lambda_{t,i}^{\zeta}(x)b_t^2
  \le
  \frac{e^{-2\tau}}{1-e^{-2\tau}}
  \bigl(1-m_t^{[i]}(x,y,\zeta)^2\bigr),
\end{align}
and, under the Doob-transformed unperturbed generator
\[
  \cL_t^{0,\zeta}h(x,y)
  =
  \frac12\sum_{i=1}^n
  \lambda_{t,i}^{\zeta}(x)\Delta_i^{xy}h(x,y),
\]
one has
\begin{align}
  \label{eq:bridge-harmonic}
  (\partial_t+\cL_t^{0,\zeta})q_t^\zeta&=0,\\
  \label{eq:bridge-square}
  (\partial_t+\cL_t^{0,\zeta})(q_t^\zeta)^2
  &=
  \frac12\sum_i\lambda_{t,i}^{\zeta}
  \bigl(\Delta_i^{xy}q_t^\zeta\bigr)^2 \notag\\
  &=
  2a_t^2\sum_i\lambda_{t,i}^{\zeta}
  |\partial_i\phi(m_t)|^2.
\end{align}
\end{lemma}

\begin{proof}
The bridge representation, the formula for \(m_t\), and the three finite
difference identities are Chen's Boolean heat bridge lemma
\cite[Lemma~9]{Chen2026}, recalled in \cref{app:chen-lemma9}, after
replacing Chen's \(T_o\) by \(\To\) and
writing \(\rho_t=e^{-(T-t)}\).  The expression for
\(\lambda_{t,i}^{\zeta}\) and the Doob-transformed unperturbed generator
are those in Chen's Doob \(h\)-transform lemma \cite[Lemma~10]{Chen2026},
recalled in \cref{app:chen-lemma10}.
The space-time harmonicity, the \(b_t\)-control, and the square-energy
identity are the corresponding bridge identities used in the proof of
Chen's weighted energy estimate \cite[Lemma~11]{Chen2026}, recalled in
\cref{app:chen-lemma11}; after the same
notation translation they give exactly
\eqref{eq:bridge-harmonic}--\eqref{eq:bridge-square}.
\end{proof}

The following consequence of \cref{lem:joint-filtration} is used repeatedly.  Even
though the joint coupling is path-dependent through
\(\1_{\{t\le\StopT\}}\), the \(V\)-coordinate has the same predictable
characteristics with respect to the enlarged filtration.  Consequently,
future terminal tests of \(V_T\) may be conditioned through the original
reverse-heat transition once \(V_t\) is fixed.

\begin{lemma}[Terminal conditioning under the perturbed coupling]\label{lem:terminal-conditioning}
Under \cref{lem:joint-filtration}, for \(\zeta\in\cube\) with
\(\Pp(V_T=\zeta)>0\), set
\(\Pp^\zeta=\Pp(\,\cdot\,\mid V_T=\zeta)\).  For
\(t\in[\theta,\To]\),
\[
  \Pp(V_T=\zeta\mid\mathcal F_t)
  =
  \Pp(V_T=\zeta\mid V_t)
  =
  H_t^\zeta(V_t),
\]
where
\[
  H_t^\zeta(x)
  =
  \frac{K_t^\zeta(x)f(\zeta)}{P_{T-t}f(x)},\qquad
  K_t^\zeta(x)=2^{-n}\prod_{j=1}^n(1+\rho_t x_j\zeta_j),
  \qquad \rho_t=e^{-(T-t)}.
\]
Let \(r_{t,i}^\zeta(x)=H_t^\zeta(\flip_i x)/H_t^\zeta(x)\).  Under
\(\Pp^\zeta\), the stopped perturbed process on \([\theta,\To]\) has
predictable generator
\[
  \cL_t^{0,\zeta}+\1_{\{t\le\StopT\}}\cB_t^\zeta,
\]
where
\[
  \cL_t^{0,\zeta}h(x,y)
  =
  \frac12\sum_i\lambda_{t,i}^{\zeta}(x)\Delta_i^{xy}h(x,y),
  \qquad
  \lambda_{t,i}^{\zeta}(x)=
  \frac{1-\rho_t x_i\zeta_i}{1+\rho_t x_i\zeta_i},
\]
and
\[
  \cB_t^\zeta h(x,y)
  =
  \sum_i\1_{\{S_i>0\}}\delta_iS_i\Delta_i^yh(x,y)
  +
  \sum_i\1_{\{S_i\le0\}}r_{t,i}^\zeta\delta_iS_i
  \Delta_i^yh(\flip_i x,y).
\]
Here \(S_i=S_i(t,x)\) and \(\delta_i=\delta_i(t,x)\).  Moreover, if
\(E\in\mathcal F_\theta\) and \(M_t^\zeta\) is a bounded martingale on
\([\theta,\To]\) under \(\Pp^\zeta\) with respect to the conditioned
filtration, then \(\1_E(M_t^\zeta-M_\theta^\zeta)\) is again a
martingale; the same statement holds under \(\Pp\).
As before, \(\cB_t^\zeta\) is a signed perturbation; only the full
operator \(\cL_t^{0,\zeta}+\1_{\{t\le\StopT\}}\cB_t^\zeta\) is the
conditioned predictable generator.
\end{lemma}

\begin{proof}
By \cref{lem:joint-filtration}, applied with
\(\Phi=\1_{\{\zeta\}}\),
\[
  \Pp(V_T=\zeta\mid\mathcal F_t)
  =
  P^V_{t,T}\1_{\{\zeta\}}(V_t)
  =
  H_t^\zeta(V_t).
\]
Bayes' formula for the reverse heat process gives, for \(x\in\cube\),
\[
  H_t^\zeta(x)
  =
  K_t^\zeta(x)\frac{f(\zeta)}{P_{T-t}f(x)},
  \qquad
  K_t^\zeta(x)=2^{-n}\prod_j(1+\rho_t x_j\zeta_j),
  \qquad \rho_t=e^{-(T-t)}.
\]
In particular the conditioning weight depends only on \(V_t\), even though
the joint coupling is path-dependent through \(\1_{\{t\le\StopT\}}\).
Since
\[
  1-2S_i(t,x)=\frac{P_{T-t}f(\flip_i x)}{P_{T-t}f(x)},
\]
we have the cancellation
\[
  r^\zeta_{t,i}(x)(1-2S_i(t,x))
  =
  \frac{H_t^\zeta(\flip_i x)}{H_t^\zeta(x)}
  \frac{P_{T-t}f(\flip_i x)}{P_{T-t}f(x)}
  =
  \frac{K_t^\zeta(\flip_i x)}{K_t^\zeta(x)}
  =
  \lambda^\zeta_{t,i}(x).
\]
Under \(\Pp^\zeta=\Pp(\,\cdot\,\mid V_T=\zeta)\), the joint generator is
the Doob transform by the space-time harmonic function
\(H_t^\zeta(V_t)\).  Hence every term that flips the \(x\)-coordinate to
\(\flip_i x\) is multiplied by \(r^\zeta_{t,i}(x)\), while a term that
flips only the \(y\)-coordinate is not.  The synchronized part therefore
becomes
\[
  \frac12\sum_i r^\zeta_{t,i}(x)(1-2S_i(t,x))\Delta_i^{xy}
  =
  \frac12\sum_i \lambda^\zeta_{t,i}(x)\Delta_i^{xy},
\]
which is \(\cL_t^{0,\zeta}\).  In the perturbation part, the \(S_i>0\)
term flips only \(y\), so no factor \(r^\zeta_{t,i}\) appears; the
\(S_i\le0\) term is attached to an \(x\)-flip and receives the factor
\(r^\zeta_{t,i}\).  This identifies the conditioned generator and gives
the displayed formula for
\(\cB_t^\zeta\).
Finally, if \(E\in\mathcal F_\theta\), then \(E\in\mathcal F_t\) for all
\(t\ge\theta\).  Since all martingales used here are bounded on a finite
state space, multiplying their increments by \(\1_E\) preserves the
martingale property, both under \(\Pp^\zeta\) and under \(\Pp\).
\end{proof}

The next estimate is the local version of the terminal testing-discrepancy
calculation: the factor \(\1_E\) is inserted before the final
Cauchy--Schwarz step.  Since \(E\in\mathcal F_\theta\), it is fixed at the
perturbation start time and may be inserted before the Duhamel identity,
before the conditional bridge expansion, and before the Doob-transform
energy estimate.  No stopping-time localization after \(\theta\) is used.

\begin{lemma}[Localized terminal testing discrepancy]\label{lem:localized-discrepancy}
Let \(E\in\mathcal F_\theta\).  Define
\[
  \mu_E^W(A)=\Pp(E,W_\To\in A),\qquad
  \mu_E^V(A)=\Pp(E,V_\To\in A).
\]
With the probability convention
\[
  d_{\TV}(\mu,\nu)=\sup_{A\subseteq\cube}|\mu(A)-\nu(A)|,
\]
define
\[
  D_E
  =
  d_{\TV}(\mu_E^W,\mu_E^V)
  =
  \sup_{\phi:\cube\to\{0,1\}}
  \left|
  \E\left[
  \1_E\{\phi(W_\To)-\phi(V_\To)\}
  \right]\right|.
\]
Thus \(D_E\) is a localized terminal testing discrepancy between two
sub-probability measures; when \(E=\Omega\), it reduces to the usual
terminal testing discrepancy between the laws of \(W_\To\) and \(V_\To\).
The two sub-probability measures have the same total mass \(\Pp(E)\), so
this convention agrees with the usual indicator-testing normalization of
total variation.
Also define
\[
  \cS_E
  =
  \E\left[
  \1_E
  \int_\theta^{\StopT}
  \bdel^2\sum_{i=1}^n S_i(t,V_{t-})^2\,dt
  \right].
\]
For \(\theta\in[\To-1,\To)\),
\[
  D_E
  \le
  C_\tau\left(\cS_E+\sqrt{\cS_E\,\Pp(E)}\right).
\]
\end{lemma}

\begin{proof}
Fix \(\phi:\cube\to\{0,1\}\).  We prove the estimate for this \(\phi\)
and then take the supremum.  Throughout the proof, \(\phi\) is identified
with its multilinear extension to \([-1,1]^n\).

\smallskip
\noindent\emph{Localized Duhamel identity.}
Let \((V_t,W_t^0)\)
denote the unperturbed synchronized joint process, started from the same pair
\((V_\theta,W_\theta)\), and let
\[
  U_t(x,y)=
  \E\bigl[\phi(W^0_\To)\mid (V_t,W_t^0)=(x,y)\bigr].
\]
This function is deterministic once \(f,T,t\) are fixed; below we evaluate
the same backward solution along the perturbed path.  Then
\(U_\To(x,y)=\phi(y)\), and \(U_t\) solves the backward equation for the
unperturbed joint generator.  More explicitly, let
\(\cL_t^0\) be the synchronized unperturbed predictable generator and let
\(\cL_t^\delta=\cL_t^0+\1_{\{t\le\StopT\}}\cB_t\) be the predictable
generator of the perturbed joint process.  Since the state space is finite
and \(t\le\To<T\), each coordinate jump rate and each finite difference
appearing below is bounded by a constant depending only on \(\tau\); for each
fixed \(n\), the total rate is finite.  Thus
\[
  U_t(V_t,W_t)-U_\theta(V_\theta,W_\theta)
  -\int_\theta^t\1_{\{s\le\StopT\}}
  \cB_s U_s(V_{s-},W_{s-})\,ds
\]
is a martingale.  Since \(E\in\mathcal F_\theta\), multiplying this
identity by \(\1_E\) preserves the martingale property.  Also, on the
unperturbed synchronized coupling, \(W_t^0=V_t\) when
\(W_\theta^0=V_\theta\).  Hence
\[
  U_\theta(V_\theta,W_\theta)
  =
  \E[\phi(V_\To)\mid\mathcal F_\theta],
  \qquad
  U_\To(V_\To,W_\To)=\phi(W_\To),
\]
and, since \(E\in\mathcal F_\theta\),
\[
  \E[\1_E U_\To(V_\To,W_\To)]
  =
  \E[\1_E\phi(W_\To)],
  \qquad
  \E[\1_E U_\theta(V_\theta,W_\theta)]
  =
  \E[\1_E\phi(V_\To)].
\]
Therefore
\begin{align}
  \label{eq:localized-duhamel}
  \E\left[\1_E(\phi(W_\To)-\phi(V_\To))\right]
  =
  \E\left[
  \1_E\int_\theta^{\StopT}
  \cB_t U_t(V_{t-},W_{t-})\,dt
  \right],
\end{align}
where \(\cB_t\) is the perturbation part of the predictable joint generator.

The perturbation leaves the \(V\)-marginal unchanged: for every test
function depending only on \(x\), the operator \(\cB_t\) vanishes.  Hence
\(V_T\) and the bridge weights \(H_t^\zeta(V_t)\) appearing below are
those of the original reverse heat process throughout the Duhamel
computation.

\smallskip
\noindent\emph{Bridge expansion.}
We next recall the bridge representation used to estimate \(\cB_t U_t\);
the coefficients \(a_t,b_t,m_t,\lambda_{t,i}^\zeta\) are those of
\cref{lem:bridge-algebra}.
For \(\zeta\in\cube\), set
\[
  K_t^\zeta(x)=2^{-n}\prod_{j=1}^n(1+\rho_t x_j\zeta_j),
  \qquad \rho_t=e^{-(T-t)}.
\]
Thus \(K_t^\zeta(x)\) is the forward heat probability of being at
\(\zeta\) at time \(T\), starting from \(x\) at time \(t\).  The actual
reverse-process bridge weight is
\[
  H_t^\zeta(x)
  =
  \Pp(V_T=\zeta\mid V_t=x)
  =
  \frac{K_t^\zeta(x)f(\zeta)}{P_{T-t}f(x)},
  \qquad
  r_{t,i}^\zeta(x)=\frac{H_t^\zeta(\flip_i x)}{H_t^\zeta(x)}.
\]
Since \(P_{T-t}f(\flip_i x)/P_{T-t}f(x)=1-2S_i(t,x)\), the heat-kernel
ratio gives the useful identity
\[
  r_{t,i}^\zeta(x)(1-2S_i(t,x))
  =
  \frac{K_t^\zeta(\flip_i x)}{K_t^\zeta(x)}
  =
  \lambda_{t,i}^\zeta(x).
\]
Under the conditional law \(\Pp^\zeta=\Pp(\,\cdot\,\mid V_T=\zeta)\), the
unperturbed synchronized generator is the Doob transform
\[
  \cL_t^{0,\zeta}h(x,y)
  =
  \frac12\sum_{i=1}^n\lambda_{t,i}^\zeta(x)\Delta_i^{xy}h(x,y),
\]
where
\[
  \lambda_{t,i}^\zeta(x)
  =
  \frac{1-\rho_t x_i\zeta_i}{1+\rho_t x_i\zeta_i}.
\]
For \(t\le\To\), the edge-ratio bound implies
\[
  C_\tau^{-1}\le \lambda_{t,i}^\zeta(x)\le C_\tau.
\]
The perturbation part under \(\Pp^\zeta\) is
\[
  \cB_t^\zeta h(x,y)
  =
  \sum_i\1_{\{S_i>0\}}\delta_iS_i\Delta_i^yh(x,y)
  +
  \sum_i\1_{\{S_i\le0\}}r_{t,i}^\zeta\delta_iS_i
  \Delta_i^yh(\flip_i x,y),
\]
where \(S_i=S_i(t,x)\) and \(\delta_i=\delta_i(t,x)\).  Moreover,
\[
  \left|\delta_i
  \bigl(\1_{\{S_i>0\}}+r_{t,i}^\zeta\1_{\{S_i\le0\}}\bigr)\right|^2
  \le C_\tau\,\bdel^2\,\lambda_{t,i}^\zeta(x).
\]
Indeed, if \(S_i>0\), then \(\delta_i=\bdel\) and
\(\lambda_{t,i}^\zeta\ge C_\tau^{-1}\).  If \(S_i\le0\), then
\[
  \delta_i=\bdel\,\frac{1-2S_i}{1-2\bdel S_i},
  \qquad
  r_{t,i}^\zeta(1-2S_i)=\lambda_{t,i}^\zeta,
\]
so
\[
  r_{t,i}^\zeta\delta_i
  =
  \bdel\,\frac{\lambda_{t,i}^\zeta}{1-2\bdel S_i}.
\]
Since \(S_i\le0\), the denominator is at least \(1\), and
\(\lambda_{t,i}^\zeta\le C_\tau\).  This gives the displayed bound.

For \(t<\To\), let \(m_t(x,y,\zeta)\in(-1,1)^n\) be the product mean of
the Boolean heat bridge from \(t\) to \(\To\), conditioned on \(V_t=x\) and
\(V_T=\zeta\), after the synchronized sign change \(x\odot y\).  Write
\[
  q_t^\zeta(x,y)=\phi(m_t(x,y,\zeta)).
\]
Equivalently, \(q_t^\zeta\) is the conditional expectation of the
\(\{0,1\}\)-valued terminal test \(\phi(W_\To^0)\) under the Boolean bridge,
so \(0\le q_t^\zeta\le1\).
The bridge formula in \cref{lem:bridge-algebra} gives
\begin{align}
  \label{eq:bridge-derivatives-local}
  \Delta_i^yq_t^\zeta(x,y)
    &=-2(a_ty_i+b_t\omega_i)\partial_i\phi(m_t(x,y,\zeta)), \notag\\
  \Delta_i^yq_t^\zeta(\flip_i x,y)
    &=-2(a_ty_i-b_t\omega_i)\partial_i\phi(m_t(x,y,\zeta)),
\end{align}
where \(\omega_i=x_iy_i\zeta_i\).  Consequently,
\[
  |\Delta_i^yq_t^\zeta(x,y)|^2+
  |\Delta_i^yq_t^\zeta(\flip_i x,y)|^2
  \le
  C(a_t^2+b_t^2)\,|\partial_i\phi(m_t(x,y,\zeta))|^2.
\]
By finite-state regular disintegration,
\[
  \E[\cdot\mid\mathcal F_t]
  =
  \sum_{\zeta\in\cube}
  \Pp(V_T=\zeta\mid\mathcal F_t)
  \E[\cdot\mid\mathcal F_t,V_T=\zeta].
\]
By \cref{lem:terminal-conditioning}, \(H_t^\zeta(V_t)\) is the
conditional law of \(V_T\) even after conditioning on the joint past
\(\mathcal F_t\), not only after conditioning on the natural filtration of
\(V\).  Therefore expanding \(U_t\) through \(V_T\) gives
\[
  U_t(x,y)=\sum_{\zeta\in\cube}H_t^\zeta(x)q_t^\zeta(x,y).
\]
Consequently the perturbative generator decomposes as
\[
  \cB_t U_t(x,y)
  =
  \sum_{\zeta\in\cube}
  H_t^\zeta(x)\cB_t^\zeta q_t^\zeta(x,y).
\]
Indeed, for each fixed \(\zeta\),
\[
  \cB_t(H_t^\zeta q_t^\zeta)(x,y)
  =
  H_t^\zeta(x)\cB_t^\zeta q_t^\zeta(x,y).
\]
For the \(S_i>0\) part, \(\cB_t\) flips only the \(y\)-coordinate, so
the coefficient \(H_t^\zeta(x)\) is unchanged.  For the \(S_i\le0\) part,
\[
  \Delta_i^y(H_t^\zeta q_t^\zeta)(\flip_i x,y)
  =
  H_t^\zeta(\flip_i x)\Delta_i^yq_t^\zeta(\flip_i x,y)
  =
  H_t^\zeta(x)r_{t,i}^\zeta(x)
  \Delta_i^yq_t^\zeta(\flip_i x,y),
\]
which is exactly the factor appearing in \(\cB_t^\zeta\).

\smallskip
\noindent\emph{Pointwise perturbation bound.}
Combining the last three displays gives
\[
  |\cB_t U_t(x,y)|
  \le C_\tau
  \sum_{\zeta}H_t^\zeta(x)\sum_i
  \bdel \abs{S_i(t,x)}\,\lambda_{t,i}^\zeta(x)^{1/2}
  (a_t^2+b_t^2)^{1/2}
  |\partial_i\phi(m_t(x,y,\zeta))|.
\]
Applying weighted Cauchy--Schwarz over the product index \((\zeta,i)\),
with weights \(H_t^\zeta(x)\), yields the pointwise estimate
\begin{align}
  \label{eq:pointwise-AtU}
  |\cB_t U_t(x,y)|
  \le
  C_\tau
  \left(\bdel^2\sum_iS_i(t,x)^2\right)^{1/2}
  \Gamma_t(x,y)^{1/2},
\end{align}
where
\[
  \Gamma_t(x,y)=
  \sum_{\zeta}H_t^\zeta(x)
  \sum_i\lambda_{t,i}^\zeta(x)(a_t^2+b_t^2)
  |\partial_i\phi(m_t(x,y,\zeta))|^2.
\]
From \eqref{eq:localized-duhamel}, \eqref{eq:pointwise-AtU}, and
Cauchy--Schwarz,
\begin{align}
  \label{eq:localized-tv-cauchy}
  \left|\E\left[\1_E(\phi(W_\To)-\phi(V_\To))\right]\right|
  \le
  C_\tau\,\cS_E^{1/2}
  \left(\E\left[\1_E\int_\theta^\To
  \Gamma_t(V_{t-},W_{t-})\,dt\right]\right)^{1/2}.
\end{align}

\smallskip
\noindent\emph{Energy closure.}
It remains to localize the bridge-gradient energy.  Define
\begin{align*}
  \Psi_a^E
  &=
  \E\left[\1_E\int_\theta^\To
  a_t^2\sum_i\lambda_{t,i}^{V_T}(V_t)
  |\partial_i\phi(m_t(V_t,W_t,V_T))|^2\,dt\right],\\
  \Psi_b^E
  &=
  \E\left[\1_E\int_\theta^\To
  b_t^2\sum_i\lambda_{t,i}^{V_T}(V_t)
  |\partial_i\phi(m_t(V_t,W_t,V_T))|^2\,dt\right].
\end{align*}
In Lebesgue-time integrals we freely replace \(V_{t-},W_{t-}\) by
\(V_t,W_t\), since the jump times are countable almost surely.
Then the energy in \eqref{eq:localized-tv-cauchy} is bounded by
\(C(\Psi_a^E+\Psi_b^E)\).
For endpoint rigor, fix \(0<\eps<\To-\theta\) and let
\(\Psi_{a,\eps}^E,\Psi_{b,\eps}^E\) denote the same quantities with the
upper limit \(\To-\eps\).  We first prove the estimates below for these
truncated energies, with constants independent of \(\eps\).  Monotone
convergence then gives the displayed full-time bounds.  To keep notation
readable, the subscript \(\eps\) is suppressed until the final limiting
step.

First consider \(\Psi_b^E\).  The bridge algebra
\eqref{eq:bridge-b-control} gives
\[
  \frac{\lambda_{t,i}^{\zeta}(x)b_t^2}
       {1-m_t^{[i]}(x,y,\zeta)^2}
  \le C_\tau.
\]
Using \cref{lem:level-one} at the point \(m_t(x,y,\zeta)\),
\[
  \sum_i(1-m_t^{[i]}(x,y,\zeta)^2)
  |\partial_i\phi(m_t(x,y,\zeta))|^2
  \le \frac14.
\]
Hence the integrand defining \(\Psi_b^E\) is at most \(C_\tau\), and
since \(\To-\theta\le1\),
\begin{align}
  \label{eq:local-psi-b}
  \Psi_b^E\le C_\tau\Pp(E).
\end{align}

We now estimate \(\Psi_a^E\).  Fix \(\zeta\) and work under \(\Pp^\zeta\).
By \eqref{eq:bridge-square},
\[
  (\partial_t+\cL_t^{0,\zeta})(q_t^\zeta)^2
  =
  \frac12\sum_i\lambda_{t,i}^{\zeta}
  \bigl(\Delta_i^{xy}q_t^\zeta\bigr)^2
  =
  2a_t^2\sum_i\lambda_{t,i}^\zeta
  |\partial_i\phi(m_t)|^2.
\]
Because \(E\in\mathcal F_\theta\), multiplying It\^o's formula for
\((q_t^\zeta(V_t,W_t))^2\) by \(\1_E\) is legitimate by
\cref{lem:terminal-conditioning}.  More explicitly, applying It\^o's
formula under \(\Pp^\zeta\), averaging over \(\zeta=V_T\), and using the
conditioned generator from \cref{lem:terminal-conditioning}, gives
\begin{align*}
  2\Psi_{a,\eps}^E
  &=
  \E\left[
  \1_E (q_{\To-\eps}^{V_T})^2(V_{\To-\eps},W_{\To-\eps})
  \right]
  -
  \E\left[
  \1_E (q_{\theta}^{V_T})^2(V_\theta,W_\theta)
  \right] \\
  &\quad -
  \E\left[
  \1_E\int_\theta^{\To-\eps}
  \1_{\{t\le\StopT\}}
  \cB_t^{V_T}(q_t^{V_T})^2(V_{t-},W_{t-})\,dt
  \right].
\end{align*}
Since \(0\le q_t^\zeta\le1\), this yields
\[
  \Psi_{a,\eps}^E
  \le
  C\Pp(E)
  +
  C\,\E\left[\1_E\int_\theta^{\To-\eps}
  \1_{\{t\le\StopT\}}
  |\cB_t^{V_T}(q_t^{V_T})^2(V_{t-},W_{t-})|\,dt\right].
\]
Since \(0\le q_t^\zeta\le1\),
\[
  |\Delta_i^y(q_t^\zeta)^2(x,y)|
  \le2|\Delta_i^yq_t^\zeta(x,y)|,\qquad
  |\Delta_i^y(q_t^\zeta)^2(\flip_i x,y)|
  \le2|\Delta_i^yq_t^\zeta(\flip_i x,y)|.
\]
Using the displayed formula for \(\cB_t^\zeta\), the estimate
\[
  \left|\delta_i
  \bigl(\1_{\{S_i>0\}}+r_{t,i}^\zeta\1_{\{S_i\le0\}}\bigr)\right|^2
  \le C_\tau\bdel^2\lambda_{t,i}^\zeta,
\]
and Cauchy--Schwarz in the coordinate \(i\), we get
\[
  |\cB_t^\zeta(q_t^\zeta)^2|
  \le
  C_\tau
  \left(\bdel^2\sum_iS_i(t,V_{t-})^2\right)^{1/2}
  \left(\Gamma_{a,t}^\zeta+\Gamma_{b,t}^\zeta\right)^{1/2},
\]
where
\[
  \Gamma_{a,t}^\zeta
  =
  a_t^2\sum_i\lambda_{t,i}^{\zeta}
  |\partial_i\phi(m_t(x,y,\zeta))|^2,\qquad
  \Gamma_{b,t}^\zeta
  =
  b_t^2\sum_i\lambda_{t,i}^{\zeta}
  |\partial_i\phi(m_t(x,y,\zeta))|^2.
\]
After evaluating at \((V_{t-},W_{t-},V_T)\), multiplying by
\(\1_E\1_{\{t\le\StopT\}}\), integrating in time, and applying
Cauchy--Schwarz, the last display gives
\[
  \Psi_a^E
  \le
  C_\tau\Pp(E)
  +
  C_\tau\cS_E^{1/2}(\Psi_a^E+\Psi_b^E)^{1/2}.
\]
Set \(Y=\Psi_a^E+\Psi_b^E\).  Combining the last display with
\eqref{eq:local-psi-b} gives
\[
  Y\le C_\tau\Pp(E)+C_\tau\cS_E^{1/2}Y^{1/2}.
\]
Using \(uv\le \frac12v^2+C_\tau u^2\), with
\(u=\cS_E^{1/2}\) and \(v=Y^{1/2}\), we obtain
\[
  Y\le C_\tau\bigl(\Pp(E)+\cS_E\bigr).
\]
Restoring the suppressed truncation parameter, this estimate is uniform
for \(Y_\eps=\Psi_{a,\eps}^E+\Psi_{b,\eps}^E\).  Letting
\(\eps\downarrow0\) and using monotone convergence gives the same bound
for the full energies \(\Psi_a^E+\Psi_b^E\).
In particular,
\[
  \Psi_a^E
  \le C_\tau\bigl(\Pp(E)+\cS_E\bigr).
\]
Plugging the bound on \(Y=\Psi_a^E+\Psi_b^E\) into
\eqref{eq:localized-tv-cauchy} gives
\[
  \left|\E\left[\1_E(\phi(W_\To)-\phi(V_\To))\right]\right|
  \le
  C_\tau\cS_E^{1/2}\bigl(\Pp(E)+\cS_E\bigr)^{1/2}
  \le
  C_\tau\left(\sqrt{\cS_E\Pp(E)}+\cS_E\right).
\]
Taking the supremum over \(\phi\) proves the lemma.
\end{proof}

\begin{corollary}[Layered discrepancy bound]\label{cor:layer-tv}
Let
\[
  G_r(\theta)=\{r\le R_\theta<r+1\},
  \qquad r\ge \alpha/2.
\]
Then, for \(r\ge \alpha/2\),
\[
  D_{G_r(\theta)}
  \le
  C_\tau
  \left(
  \frac{\alpha}{\sqrt r}+\frac{\alpha^2}{r}
  \right)\Pp(G_r(\theta)).
\]
More generally, the same estimate holds with \(G_r(\theta)\) replaced by
any \(E\in\mathcal F_\theta\) such that \(E\subset G_r(\theta)\).
Moreover, for \(G_{\ge L/2}(\theta)=\{R_\theta\ge L/2\}\),
\[
  D_{G_{\ge L/2}(\theta)}
  \le
  C_\tau\left(\frac{\alpha}{\sqrt L}+\frac{\alpha^2}{L}\right).
\]
\end{corollary}

\begin{proof}
Let \(E\in\mathcal F_\theta\) with \(E\subset G_r(\theta)\).  The active
part is \(E\cap\{R_\theta\ge\alpha\}\).  On this active part,
\[
  \bdel^2(R_\theta+\alpha+1)
  =
  \frac{\alpha^2(R_\theta+\alpha+1)}{(R_\theta+1)^2}
  \le
  C\frac{\alpha^2}{r},
  \qquad r\ge\alpha/2,
\]
while outside it \(\bdel=0\).  By \cref{lem:score-energy},
\[
  \cS_E
  \le
  C_\tau\frac{\alpha^2}{r}\Pp(E).
\]
Plugging this into \cref{lem:localized-discrepancy} gives the subset
estimate, and taking \(E=G_r(\theta)\) gives the first displayed bound.  On
\(E=G_{\ge L/2}(\theta)\), since \(L\ge8\) gives \(\alpha\le L/2\), the
perturbation is active and
\[
  \bdel^2(R_\theta+\alpha+1)
  =
  \frac{\alpha^2(R_\theta+\alpha+1)}{(R_\theta+1)^2}
  \le
  \frac{C \alpha^2}{R_\theta+1}
  \le
  \frac{C \alpha^2}{L}.
\]
Together with \cref{lem:score-energy} this gives
\(\cS_E\le C_\tau \alpha^2\Pp(E)/L\), and the second estimate follows from
\cref{lem:localized-discrepancy} and \(\Pp(E)\le1\).  No layer-profile estimate is
needed for this tail layer: it appears only once in the final
decomposition, and the crude bound \(\Pp(E)\le1\) already gives
\(O_\tau(\alpha/\sqrt L)\).
\end{proof}

\section{Proof of fixed-time anti-concentration}\label{sec:proof-fixed-time}

In this section, we apply our layerwise localization approach given in the preceding section to complete the proof of Proposition~\ref{prop:AC}.

\begin{proof}[Proof of \cref{prop:AC}]
Fix \(L\ge8\), and put \(\alpha=\frac12\log L+1\).  For a fixed
\(\theta\in[\To-1,\To)\), first compare the two tails at level \(L\).
Since \(\bdel=0\) on \(\{R_\theta<\alpha\}\), the perturbation term is
identically zero after \(\theta\) on this event and the two synchronized
processes coincide up to \(\To\).  Hence, for
\(\varphi=\1_{\{g_\tau>L\}}\),
\[
  \Pp\{g_\tau(V_\To)>L\}
  \le
  \Pp\{g_\tau(W_\To)>L\}+D_{\{R_\theta\ge \alpha\}}.
\]
Since
\[
  \cA_\tau((L,L+1])
  =
  \Pp\{g_\tau(V_\To)>L\}
  -
  \Pp\{g_\tau(V_\To)>L+1\},
\]
combining the last display with \cref{lem:monotone} gives the bootstrap
inequality in one line:
\begin{align*}
  \cA_\tau((L,L+1])
  &\le
  \Pp\{g_\tau(W_\To)>L\}
  +D_{\{R_\theta\ge \alpha\}}
  -
  \Pp\{g_\tau(V_\To)>L+1\} \\
  &\le
  D_{\{R_\theta\ge \alpha\}}
  +\cA_{T-\theta}((L-\alpha,L+\alpha])
  +\frac{3}{\sqrt L}.
\end{align*}
Decompose
\[
  \{R_\theta\ge \alpha\}
  =
  \left(\bigcup_{r=\lfloor \alpha\rfloor}^{\lfloor L/2\rfloor}
  E_r(\theta)\right)
  \cup G_{\ge L/2}(\theta),
  \qquad
  E_r(\theta)=G_r(\theta)\cap\{\alpha\le R_\theta<L/2\}.
\]
This union is disjoint.  The intersection in the definition of
\(E_r(\theta)\) only trims the boundary pieces near \(R_\theta=\alpha\) and
\(R_\theta=L/2\); the subset version of \cref{cor:layer-tv} is included
precisely for this use.
For any disjoint family \((F_j)\subset\mathcal F_\theta\), the definition
of \(D_E\) and the triangle inequality give
\[
  D_{\bigcup_jF_j}
  \le \sum_j D_{F_j}.
\]
Indeed, this holds for each fixed test function \(\phi\), and then one
takes the supremum over \(\phi\).
By \cref{cor:layer-tv},
\[
  D_{\{R_\theta\ge \alpha\}}
  \le
  C_\tau
  \sum_{r=\lfloor \alpha\rfloor}^{\lfloor L/2\rfloor}
  \left(\frac{\alpha}{\sqrt r}+\frac{\alpha^2}{r}\right)\Pp(E_r(\theta))
  +
  C_\tau\left(\frac{\alpha}{\sqrt L}+\frac{\alpha^2}{L}\right).
\]
Averaging \(\theta\) over \([\To-1,\To]\), and writing \(s=T-\theta\),
we have \(s\in[\tau,\tau+1]\).  For
\(\lfloor \alpha\rfloor\le r\le L/2\),
the identity below follows from
\(R_\theta=[L-g_{T-\theta}(V_\theta)]_+\): on \(G_r(\theta)\), one has
\(L-r-1<g_{T-\theta}(V_\theta)\le L-r\).  Hence
\[
  \int_{\To-1}^{\To}\Pp(E_r(\theta))\,d\theta
  \le
  \int_{\To-1}^{\To}\Pp(G_r(\theta))\,d\theta
  =
  \int_\tau^{\tau+1}
  \cA_s((L-r-1,L-r])\,ds
  \le \frac{C}{L-r-1}
  \le \frac{C}{L}
\]
by \cref{lem:time-profile}, since
\(L-r-1\ge L/2-1\ge L/4\) for \(L\ge8\).  Therefore
\begin{align*}
  \int_{\To-1}^{\To}D_{\{R_\theta\ge \alpha\}}\,d\theta
  &\le
  C_\tau
  \frac{\alpha}{L}\sum_{r=\lfloor \alpha\rfloor}^{\lfloor L/2\rfloor}r^{-1/2}
  +
  C_\tau
  \frac{\alpha^2}{L}\sum_{r=\lfloor \alpha\rfloor}^{\lfloor L/2\rfloor}r^{-1}
  +
  C_\tau\frac{\alpha}{\sqrt L}                                      \\
  &\le
  C_\tau\frac{\alpha}{\sqrt L}.
\end{align*}
Here the first sum is \(O(\sqrt L)\), while the second is \(O(\log L)\);
since \(\alpha=\frac12\log L+1\), the term
\[
  \frac{\alpha^2}{L}\sum_{r=\lfloor \alpha\rfloor}^{\lfloor L/2\rfloor}r^{-1}
  \le
  C\frac{\alpha^2\log L}{L}
  =
  C\frac{\alpha}{\sqrt L}\cdot\frac{\alpha\log L}{\sqrt L}
  \le
  C\frac{\alpha}{\sqrt L}.
\]
The last inequality holds because, for \(L\) larger than an absolute
constant, \(\alpha\log L\le C\sqrt L\), while the remaining compact range
\(8\le L\le L_0\) is absorbed into the constant.
The initial-band term is bounded by covering \((L-\alpha,L+\alpha]\) with the
half-open unit intervals \((m,m+1]\) which intersect it; there are at most
\(2\alpha+2\) of them.  Since \(L\ge8\), every such interval has
\(m\ge L-\alpha-1>2\), so these unit intervals lie in the range of
\cref{lem:time-profile}.  Thus
\[
  \int_{\To-1}^{\To}
  \cA_{T-\theta}((L-\alpha,L+\alpha])\,d\theta
  \le
  \frac{C \alpha}{L-\alpha}
  \le
  C\frac{\alpha}{L}
  \le C\frac{\alpha}{\sqrt L}.
\]
Also \(3/\sqrt L\le3\alpha/\sqrt L\), since \(\alpha\ge1\).  Thus the averaged
bootstrap inequality gives
\[
  \cA_\tau((L,L+1])
  \le C_\tau\frac{\alpha}{\sqrt L}
  \le C_\tau\frac{\log L}{\sqrt L}.
\]
\end{proof}

\section{Discussion}\label{sec:discussion}

\paragraph{Where the layerwise gain enters.}
Without localization, the terminal-discrepancy estimate controls the
perturbation by
\[
  \alpha\sqrt{\E\frac{\1_{\{R_\theta\ge \alpha\}}}{R_\theta+1}}.
\]
After time averaging, the expectation is bounded by \((\log L)/L\), which
yields a \((\log L)^{3/2}/\sqrt L\) fixed-time scale.  The localized
estimate keeps the layer event \(G_r(\theta)\) inside the Duhamel identity
and the terminal bridge estimate.  It replaces this global term by the
layer sum
\[
  \sum_r\frac{\alpha}{\sqrt r}\Pp(G_r),
\]
and the same time-smoothed profile estimate leads to
\[
  \frac{\alpha}{L}\sum_{r\le L/2}r^{-1/2}
  \lesssim \frac{\alpha}{\sqrt L}.
\]

\paragraph{Origin of the remaining logarithm.}\label{par:layerwise-limitation}
Put \(L=\log\eta\) and \(\alpha\simeq\log L\).  After averaging the start time
\(\theta\in[\To-1,\To]\), the proof uses the layer estimate
\[
  \int_{\To-1}^{\To}D_{G_r(\theta)}\,d\theta
  \lesssim
  \left(\frac{\alpha}{\sqrt r}+\frac{\alpha^2}{r}\right)\overline p_r,
  \qquad
  \overline p_r:=
  \int_{\To-1}^{\To}\Pp(G_r(\theta))\,d\theta,
\]
and the time-smoothed profile consequence
\[
  \overline p_r
  \lesssim
  \frac1{L-r-1},
  \qquad \alpha\le r\le L/2.
\]
On the middle range \(L/4\le r\le L/2\), this gives
\(\overline p_r\lesssim L^{-1}\).  The abstract layer profile
\[
  \overline p_r=\frac{c}{L},\qquad L/4\le r\le L/2,
  \qquad
  \overline p_r=0\quad\text{otherwise},
\]
with \(c>0\) small, is compatible with the profile information used in the
argument.  If the localized Cauchy--Schwarz estimates are saturated on
these layers, the contribution allowed by the method is
\[
  \sum_{L/4\le r\le L/2}
  \frac{\alpha}{\sqrt r}\frac{c}{L}
  \asymp
  \frac{\alpha}{L}\sum_{L/4\le r\le L/2}r^{-1/2}
  \asymp
  \frac{\alpha}{\sqrt L}.
\]
Since \(\alpha\simeq\log L\), this is the fixed-time
\((\log L)/\sqrt L\) scale, equivalently the weak-type factor
\[
  \frac{\log\log\eta}{\eta\sqrt{\log\eta}}.
\]
Within this layerwise Cauchy--Schwarz framework, an improvement would have
to rule out this middle-layer occupation profile or reduce the
\(\alpha/\sqrt r\) cost on those layers.

\appendix

\section{Additional lemmas quoted in the paper}\label{app:chen-inputs}

This appendix records the statements from Chen \cite{Chen2026} that are quoted
in the body of the paper but not already displayed there.  They are written in
the notation of this paper.  In particular \(L=\log\eta\),
\(\alpha=\frac12\log L+1\),
\[
  R_\theta=[L-g_{T-\theta}(V_\theta)]_+,\qquad
  I_\theta=\E\frac{\1_{\{R_\theta\ge\alpha\}}}{R_\theta+1},
\]
and Chen's constant \(\kappa=(1+e^{-\tau})/(1-e^{-\tau})\) is absorbed into
constants denoted by \(C_\tau\) in the main text.  When \(f\) is evaluated at
a point of \((-1,1)^n\), it denotes the multilinear extension of the function
on the cube.

\begin{lemma}[Chen's TV distance control {\cite[Lemma~2]{Chen2026}}]
\label{app:chen-lemma2}
Let \(\eta>e^3\), \(f:\cube\to(0,\infty)\) with \(\|f\|_1=1\), and let
\((V_t,W_t)_{t\in[\theta,T]}\) be Chen's perturbed reverse-heat coupling.
Then
\[
  d_{\mathrm{TV}}\bigl(\Law(V_\To),\Law(W_\To)\bigr)
  \lesssim
  \kappa^2\alpha^2 I_\theta
  +
  \sqrt{\kappa^2\alpha^2 I_\theta}\,
  \sqrt{1+\frac{e^{-2\tau}}{1-e^{-2\tau}}(\To-\theta)}.
\]
\end{lemma}

\begin{lemma}[Chen's time-smoothed profile bound {\cite[Lemma~4]{Chen2026}}]
\label{app:chen-lemma4}
For \(f:\cube\to(0,\infty)\) with \(\|f\|_1=1\), and for every \(\ell>2\),
\[
  \int_0^\infty \cA_s((\ell,\ell+1])\,ds
  \lesssim \frac1\ell .
\]
\end{lemma}

\begin{lemma}[Chen's edge-ratio bound {\cite[Lemma~5]{Chen2026}}]
\label{app:chen-lemma5}
Let \(f:\cube\to\R_+\) be nonzero and let the same letter denote its
multilinear extension.  For every \(x\in(-1,1)^n\),
\[
  0<\|f\|_1\prod_{j=1}^n(1-\abs{x_j})
  \le f(x)
  \le
  \|f\|_1\prod_{j=1}^n(1+\abs{x_j}).
\]
For every \(i\in[n]\),
\[
  \frac{1-\abs{x_i}}{1+\abs{x_i}}
  \le
  \frac{f(\flip_i x)}{f(x)}
  \le
  \frac{1+\abs{x_i}}{1-\abs{x_i}}.
\]
If \(f\) is \(\{0,1\}\)-valued, then its multilinear extension takes values
in \([0,1]\) on \((-1,1)^n\).
\end{lemma}

\begin{lemma}[Chen's existence and perturbation bound {\cite[Lemma~6]{Chen2026}}]
\label{app:chen-lemma6}
The coupled jump equation defining \((V_t,W_t)_{t\in[\theta,T]}\) has a
unique strong solution.  Moreover, for \(t\le\To\) and \(i\in[n]\),
\[
  0\le \delta_i(t,V_{t-})<\kappa,
  \qquad
  \kappa=\frac{1+e^{-\tau}}{1-e^{-\tau}}.
\]
\end{lemma}

\begin{lemma}[Chen's predictable generator representation {\cite[Eq.~(19)--(20)]{Chen2026}}]
\label{app:chen-generator}
Conditionally on \(\mathcal F_\theta\), the coupled process is a
finite-state pure-jump process with predictable generator
\[
  \bar{\cL}_t^\delta
  =
  \bar{\cL}_t^0+\1_{\{t\le\StopT\}}\cB_t,
\]
where, writing \(S_i=S_i(t,x)\) and \(\delta_i=\delta_i(t,x)\),
\begin{align*}
  \bar{\cL}_t^0h(x,y)
  &=
  \frac12\sum_i(1-2S_i)\Delta_i^{xy}h(x,y),\\
  \cB_t h(x,y)
  &=
  \sum_i\1_{\{S_i>0\}}\delta_iS_i\Delta_i^yh(x,y)
  +
  \sum_i\1_{\{S_i\le0\}}\delta_iS_i
  \Delta_i^yh(\flip_i x,y).
\end{align*}
This is Chen's representation \cite[Eq.~(19)--(20)]{Chen2026}, translated to
the notation of \cref{eq:joint-generator}.
\end{lemma}

\begin{lemma}[Chen's Boolean heat bridge {\cite[Lemma~9]{Chen2026}}]
\label{app:chen-lemma9}
Fix \(0\le t\le\To\) and \(x,y,\zeta\in\cube\).  Conditional on
\((V_t,V_T)=(x,\zeta)\), the coordinates of \(V_\To\) are independent after
the standard Boolean bridge conditioning, and
\[
  m_t(x,y,\zeta)
  :=
  \E[V_\To\odot(x\odot y)\mid V_t=x,\ V_T=\zeta]
\]
has coordinates
\[
  m_t^{[i]}(x,y,\zeta)=a_ty_i+b_t x_iy_i\zeta_i,
\]
where
\[
  a_t=\frac{\sinh(T-\To)}{\sinh(T-t)},\qquad
  b_t=\frac{\sinh(\To-t)}{\sinh(T-t)}.
\]
For every multilinear \(\phi\),
\[
  \E[\phi(V_\To\odot(x\odot y))\mid V_t=x,\ V_T=\zeta]
  =
  \phi(m_t(x,y,\zeta)).
\]
If \(q_t^\zeta(x,y)=\phi(m_t(x,y,\zeta))\), then
\begin{align*}
  \Delta_i^yq_t^\zeta(x,y)
  &=-2(a_ty_i+b_tx_iy_i\zeta_i)\partial_i\phi(m_t(x,y,\zeta)),\\
  \Delta_i^yq_t^\zeta(\flip_i x,y)
  &=-2(a_ty_i-b_tx_iy_i\zeta_i)\partial_i\phi(m_t(x,y,\zeta)),\\
  \Delta_i^{xy}q_t^\zeta(x,y)
  &=-2a_ty_i\partial_i\phi(m_t(x,y,\zeta)).
\end{align*}
\end{lemma}

\begin{lemma}[Chen's Doob \(h\)-transform {\cite[Lemma~10]{Chen2026}}]
\label{app:chen-lemma10}
For \(\zeta\in\cube\), define
\[
  H_t^\zeta(x)=\Pp(V_T=\zeta\mid V_t=x),\qquad
  r_{t,i}^\zeta(x)=\frac{H_t^\zeta(\flip_i x)}{H_t^\zeta(x)},\qquad
  \lambda_{t,i}^\zeta(x)=
  \frac{1-\rho_t x_i\zeta_i}{1+\rho_t x_i\zeta_i}.
\]
Under the conditioned law \(\Pp^\zeta=\Pp(\,\cdot\,\mid V_T=\zeta)\), the
Doob-transformed predictable generator has the form
\[
  \cG_t^\zeta
  =
  \cL_t^{0,\zeta}+\1_{\{t\le\StopT\}}\cB_t^\zeta,
\]
where, with \(S_i=S_i(t,x)\), \(\delta_i=\delta_i(t,x)\),
\(r_i^\zeta=r_{t,i}^\zeta(x)\), and
\(\lambda_i^\zeta=\lambda_{t,i}^\zeta(x)\),
\begin{align*}
  \cL_t^{0,\zeta}h(x,y)
  &=
  \sum_i r_i^\zeta\left(\frac12-S_i\right)\Delta_i^{xy}h(x,y),\\
  \cB_t^\zeta h(x,y)
  &=
  \sum_i\1_{\{S_i>0\}}\delta_iS_i\Delta_i^yh(x,y)
  +
  \sum_i\1_{\{S_i\le0\}}r_i^\zeta\delta_iS_i
  \Delta_i^yh(\flip_i x,y).
\end{align*}
Moreover,
\[
  r_i^\zeta\left(\frac12-S_i\right)=\frac12\lambda_i^\zeta .
\]
\end{lemma}

\begin{lemma}[Chen's weighted energy estimate {\cite[Lemma~11]{Chen2026}}]
\label{app:chen-lemma11}
Let \(\phi:\cube\to\{0,1\}\), and let \(a_t,b_t,m_t\) and
\(\lambda_{t,i}^\zeta\) be as in \cref{app:chen-lemma9,app:chen-lemma10}.
Define
\begin{align*}
  \Psi_a
  &:=
  \E\int_\theta^\To
  a_t^2\sum_i
  \lambda_{t,i}^{V_T}(V_t)
  \abs{\partial_i\phi(m_t(V_t,W_t,V_T))}^2\,dt,\\
  \Psi_b
  &:=
  \E\int_\theta^\To
  b_t^2\sum_i
  \lambda_{t,i}^{V_T}(V_t)
  \abs{\partial_i\phi(m_t(V_t,W_t,V_T))}^2\,dt .
\end{align*}
Then, with
\[
  \mathcal S_{\mathrm{Ch}}
  :=
  \E\int_\theta^{\StopT}
  \sum_i\bdel^2 S_i(t,V_{t-})^2\,dt,
\]
we have
\[
  \Psi_b
  \le
  \frac{e^{-2\tau}}{1-e^{-2\tau}}(\To-\theta),
  \qquad
  \Psi_a
  \lesssim
  1+\kappa\mathcal S_{\mathrm{Ch}}+\Psi_b.
\]
\end{lemma}

\begingroup
\small
\bibliographystyle{plain}
\bibliography{reference}

@article{BallBartheBednorzOleszkiewiczWolff2013,
  author = {Ball, K. and Barthe, F. and Bednorz, W. and Oleszkiewicz, K. and Wolff, P.},
  title = {{$L^1$}-smoothing for the {Ornstein--Uhlenbeck} semigroup},
  journal = {Mathematika},
  volume = {59},
  number = {1},
  pages = {160--168},
  year = {2013}
}

@misc{Chen2026,
  author = {Chen, Yuansi},
  title = {Talagrand's convolution conjecture up to loglog via perturbed reverse heat},
  year = {2025},
  eprint = {2511.19374v2},
  archivePrefix = {arXiv},
  primaryClass = {math.PR},
  note = {arXiv:2511.19374v2. Version 2, revised May 1, 2026}
}

@article{EldanLee2018,
  author = {Eldan, Ronen and Lee, James R.},
  title = {Regularization under diffusion and anti-concentration of the information content},
  journal = {Duke Mathematical Journal},
  volume = {167},
  number = {5},
  pages = {969--993},
  year = {2018}
}

@article{Lehec2016,
  author = {Lehec, Joseph},
  title = {Regularization in {$L^1$} for the {Ornstein--Uhlenbeck} semigroup},
  journal = {Annales de la Facult{\'e} des sciences de Toulouse: Math{\'e}matiques},
  volume = {25},
  pages = {191--204},
  year = {2016}
}

@book{ODonnell2014,
  author = {O'Donnell, Ryan},
  title = {Analysis of Boolean Functions},
  publisher = {Cambridge University Press},
  year = {2014}
}

@article{Talagrand1989,
  author = {Talagrand, Michel},
  title = {A conjecture on convolution operators, and a non-{Dunford--Pettis} operator on {$L^1$}},
  journal = {Israel Journal of Mathematics},
  volume = {68},
  pages = {82--88},
  year = {1989}
}

@misc{Talagrand2016,
  author = {Talagrand, Michel},
  title = {Regularization from {$L^1$} by convolution},
  howpublished = {\url{https://michel.talagrand.net/prizes/convolution.pdf}},
  year = {2016}
}
\endgroup

\end{document}